\documentclass[11pt]{article}
\usepackage{amsmath}
\usepackage{amsthm}
\usepackage{amssymb}
\usepackage{amscd}
\usepackage[dvipdfm]{hyperref}
\newcommand{\dbar}{\bar{\partial}}
\newcommand{\C}{\mathbb{C}}
\newcommand{\R}{\mathbb{R}}

\theoremstyle{plain}
\newtheorem{theorem}{Theorem}[section]
\newtheorem{proposition}[theorem]{Proposition}
\newtheorem{lemma}[theorem]{Lemma}
\newtheorem{Corollary}[theorem]{Corollary}

\theoremstyle{definition}
\newtheorem{definition}[theorem]{Definition}
\newtheorem{remark}[theorem]{Remark}

\makeatletter
\@addtoreset{equation}{section}
\makeatother
\newcounter{constantLABEL}
\newcommand{\cref}[1]{C_{\ref{#1}}}

\newcounter{constantslabel}
\begin{document}

%%%%%%%%%%%%%%%%%%%%%%%%%%%%%%%%%%%%%%%%%%%%%%%%%%%%%%%%%%%%%%%%%%%%%%%%%%%%%%%%%%%%%%%%%%%%%%%%%%%%%

\title{On the moduli space of Donaldson--Thomas instantons} 

%%%%%%%%%%%%%%%%%%%%%%%%%%%%%%%%%%%%%%%%%%%%%%%%%%%%%%%%%%%%%%%%%%%%%%%%%%%%%%%%%%%%%%%%%%%%%%%%%%%%%

\author{Yuuji Tanaka}

\date{}

%%%%%%%%%%%%%%%%%%%%%%%%%%%%%

\maketitle

%%%%%%%%%%%%%%%%%%%%%%%%%%%%%%%%%%%%%%%%%%%%%%%%%%%%%%%%%%%%%%%%%%%%%%%%%%%%%%%%%%%%%%%%%%%%%%%%%%%%%

\begin{abstract}

In alignment with a programme by Donaldson and Thomas \cite{DoTh98}, 
Thomas \cite{Thomas00} 
constructed 
a deformation invariant for smooth projective 
Calabi--Yau threefolds,  which is now called 
the Donaldson--Thomas invariant, 
from the moduli space of (semi-)stable
sheaves by using algebraic geometry techniques.

In the same paper \cite{Thomas00}, Thomas noted that certain perturbed
 Hermitian--Einstein equations might possibly 
produce an analytic theory of
the invariant. 
This article sets up the equations on symplectic 6-manifolds, 
and gives the local model and structures of the moduli 
space coming from the equations. 
We then describe a Hitchin--Kobayashi style correspondence for the equations 
on compact K\"{a}hler threefolds, which turns out to be a special case of results by 
\'{A}lvarez-C\'{o}nsul and Garc\'{\i}a-Prada \cite{AG}.

\end{abstract}

%%%%%%%%%%%%%%%%%%%%%%%%%%%%%%%%%%%%%%%%%%%%%%%%%%%%%%%%%%%%%%%%%%%%%%%%%%%%%%%%%%%%%%%%%%%%%%%%%%%%%%

\markboth{}
{}


\section{Introduction}

In \cite{DoTh98},  Donaldson and Thomas suggested 
higher-dimensional
analogues of gauge theories, 
and proposed the following two directions:  
gauge theories on $Spin(7)$ and $G_2$-manifolds; 
and gauge theories in complex 3 and 4 dimensions.
The first ones could be related to ``Topological M-theory'' proposed by 
Nekrasov and others   
\cite{MR2121727}, \cite{Dijkgraaf:2004te}.  
The second ones are a ``complexification'' of the lower-dimensional gauge
theories.  
In this direction, Thomas \cite{Thomas00} 
constructed a deformation invariant of smooth projective 
Calabi--Yau threefolds from the moduli space of (semi-)stable sheaves, 
which he called the {\it holomorphic Casson
invariant} because it can be viewed as a complex analogue of
the Taubes--Casson invariant \cite{MR1037415}. 
It is now called the
Donaldson--Thomas invariant (D--T invariant for short),  
and further developed by Joyce--Song
\cite{JS} and Kontsevich--Soibelman \cite{KS}, \cite{KS2}, \cite{KS3}. 
Later, Donaldson and Segal \cite{DS0902} further promoted the programme, 
taking into account the progress made after the proposal.  
Recently, more breakthroughs concerning the ``categorification'' of the
D--T invariant by using perverse sheaves were made by a group led by Joyce 
\cite{GJ1}, \cite{GJ2}, \cite{GJ3}, \cite{GJ4}, \cite{BBBJ}, 
also by Kiem--Li \cite{KL}.

Let us mention here a conjecture (called the MNOP
conjecture) posed by 
 Maulik--Nekrasov--Okounkov--Pandharipande \cite{MNOP1}, \cite{MNOP2}, 
which insists that the  rank one D--T invariants (``counting'' of ideal
sheaves on a Calabi--Yau threefold) can be 
determined by only 
the Betti numbers and the Gromov--Witten invariants. 
Assuming the conjecture is
true, 
one can observe that 
the rank one D--T invariants are symplectic invariants, 
as the Gromov-Witten invariants are symplectic invariants.  
One might further speculate that the full D--T
invariants defined by Joyce and Song could be also symplectic
invariants. 
One of our goals is 
to work toward proving this by 
using a gauge-theoretic equation (we call it the {\it Donaldson--Thomas
equation}) on a compact symplectic 6-manifold,
which ought to be an analytic counterpart of the notion of stable holomorphic vector
bundles, as the problem is analytic in nature.

Perhaps, one might think of that a gauge-theoretic equation which
would describe the D--T 
invariant could be the Hermitian--Einstein equations,  
as the Hitchin--Kobayashi correspondence \cite{D2}, \cite{D3}, \cite{UY1}, \cite{UY2} 
(see also \cite{MR909698}, \cite{LT}) insists that  
there is a one-to-one correspondence between 
the existence of the Hermitian--Einstein
connection  
and the Mumford--Takemoto stability 
of an irreducible vector bundle over a compact K\"{a}hler
manifold. 
However, the Hermitian--Einstein equations do not form an elliptic system
even with a gauge fixing equation in complex dimension three and more (see Section
2.1), so
this might cause a little problem.

In order to work out this issue, 
Donaldson and Thomas \cite{Thomas00} 
suggested a perturbation of the Hermitian--Einstein 
equations described below. 
This perturbation was also brought in by Baulieu--Kanno--Singer
\cite{BaKaSi98} and Iqbal--Nekrasov--Okounkov--Vafa \cite{IqNeOkVa03} in
String Theory context.

Let $Z$ be a compact symplectic 6-manifold with symplectic form
$\omega$, 
$P$ a principal
 $U(r)$-bundle on $Z$, and $E$ the associated unitary vector bundle on
 $Z$. 
The equations we consider are ones for a 
 connection $A$ of $P$ and an $\text{Ad} (P)$-valued (0,3)-form $u$ on $Z$
 of the following form. 
\begin{equation*}
 F_{A}^{0,2} + \bar{\partial}_{A}^{*} u = 0 , \qquad 
 F_{A}^{1,1} \wedge \omega^2 + [ u , \bar{u} ] +  
 2 \pi i \mu (E)  Id_{E} \omega^3 =0, 
\label{DTeq}
\end{equation*} 
where $F_{A}^{0,2}$ and $F_{A}^{1,1}$ are the (0,2) and (1,1) components
 of the curvature $F_{A}$ of $A$, and $\mu(E) := \frac{1}{r} \int_{Z} c_1 (E) \wedge \omega^2$. 
Here we picked up an almost complex structure compatible with $\omega$
 to get the splitting of the space of the complexified two forms. 
We call the equations the {\it Donaldson--Thomas equations} ({\it D--T
equations} for short) and a solution to the equations a {\it Donaldson--Thomas
 instanton} ({\it D--T instanton} for short). 
These equations with a gauge fixing equation form an elliptic system. 
We aim at developing an analytic theory concerning the D--T invariant by
 using the moduli space coming from these equations.

In \cite{Tan3}, \cite{Tan4}, we studied some analytic properties of
solutions to the equations on compact K\"{a}hler threefolds. 
In \cite{Tan3}, we proved that a sequence of solutions to the 
D--T equation has a subsequence which smoothly converges 
to a solution to the D--T equation outside a closed 
subset of the Hausdorff dimension two. 
In \cite{Tan4}, we proved 
some of singularities which appeared in the above weak limit can be
removed.

In this article, we describe the infinitesimal deformation 
and the Kuranishi model of the moduli space of D--T 
instantons by using familiar techniques in gauge theory,  
for example, the corresponding results 
for the anti-self-dual instantons in real four
dimensions were studied by Atiyah--Hitchin--Singer \cite{MR506229} 
(see also \cite{FrUh91}, \cite{DK}),  
and for the Hermitian--Einstein connections by  Kim  
\cite{MR888136} (see also \cite{MR909698}, \cite{LT}). 
We then describe 
a Hitchin--Kobayashi style correspondence for 
the D--T instanton on compact K\"{a}hler threefolds, 
which turns out to be a special case of results by 
\'{A}lvarez-C\'{o}nsul and Garc\'{\i}a-Prada \cite{AG}.

The organisation of this article is as follows. 
In Section \ref{sec:DTinst}, 
we briefly recall the Hermitian--Einstein 
connections,   
subsequently,  
we introduce the D--T equations on symplectic 6-manifolds. 
We also mention a relation between 
the D--T equations and the complex anti-self-dual
equations by dimensional reduction argument. 
In Section \ref{sec:KDT}, we give the Kuranishi model of the space of the 
D--T instantons. 
In Section \ref{sec:HKDT}, we describe a Hitchin--Kobayashi style correspondence for 
the D--T instanton on compact K\"{a}hler threefolds.

\paragraph{Acknowledgements.}
I would like to thank Mikio Furuta, 
Ryushi Goto, Ryoichi Kobayashi, Hiroshi Ohta for 
valuable comments, and referees for many useful advice. 
I am also grateful to Katrin Wehrheim 
for  wonderful encouragement. 
A part of this article was written when I had visited Beijing
International Center for Mathematical Research, Peking University  
in 2008--2009, I am very grateful to Gang Tian and the institute for
their support and hospitality.  
A part of revision was made during my visit to Institut des Hautes \'{E}tudes Scientifiques in February to March of 2012. 
I would like to thank the institute for the support and giving me an excellent research environment.  
Last but not least, I would like to thank Dominic
Joyce for enlightening me on these subjects over the years. 
This work was partially supported by JSPS Grant-in-Aid for Scientific Research No. 15H02054.

\section{The Donaldson--Thomas instantons}
\label{sec:DTinst}

\subsection{The Hermitian--Einstein connections on
compact K\"{a}hler manifolds}

We first recall the notion of the Hermitian--Einstein connections on
compact K\"{a}hler manifolds. 
General references for the Hermitian--Einstein connections are 
\cite{MR909698} and \cite{LT}.

Let $X$ be a compact K\"{a}hler manifold of complex dimension $n$ with
K\"{a}hler form $\omega$, $E$ a hermitian vector bundle over $X$ with
hermitian metric $h$. 
A metric preserving connection $A$ of $E$ is said to be a 
\textit{Hermitian--Einstein connection} if $A$ satisfies the following
equations. 
\begin{gather}
 F_{A}^{0,2} =0,  \quad 
i \Lambda F_{A}^{1 ,1}   = 2 n \pi \mu(E) Id_{E}  ,
\label{HE2}
\end{gather}
where $F_{A}^{0,2}$ and $F_{A}^{1,1}$ are the (0,2) and (1,1) components
 of the curvature $F_{A}$ of $A$, $\Lambda := (\omega)^{*}$, and $\mu (E) := \frac{1}{r} \int_X c_1 (E) \wedge \omega^{n-1}$.

The existence of a solution to the equations \eqref{HE2} is related to the notion of stability for holomorphic vector bundles. 
In fact, Donaldson \cite{D2}, \cite{D3} and Uhlenbeck--Yau \cite{UY1}, \cite{UY2} proved that 
there is a one-to-one correspondence between 
the existence of the Hermitian--Einstein
connection  
and the Mumford--Takemoto stability 
of an irreducible vector bundle over a compact K\"{a}hler
manifold (see also \cite{MR909698}, \cite{LT}).

The infinitesimal deformation of a Hermitian--Einstein connection $A$
was 
studied by Kim \cite{MR888136}  (see also \cite{MR909698}, \cite{Reyes_Carrion98}), 
and it is described by the following.  
\begin{equation}
\begin{split}
 0 \longrightarrow \Omega^{0} ( X, &\mathfrak{u}(E)) \xrightarrow{d_{A}} 
 \Omega^1 ( X, \mathfrak{u}(E)) \xrightarrow{d_{A}^{+}} 
 \Omega^{+} (X,  \mathfrak{u}(E)) \\
  & \qquad \qquad \xrightarrow{\bar{D}_{A}'} 
A^{0,3} ( X, \mathfrak{u} (E)) \xrightarrow{\bar{D}_{A}} 
A^{0,4} ( X, \mathfrak{u} (E))  \\ 
 &\qquad  \qquad \qquad \quad  \xrightarrow{\bar{D}_{A}} 
\cdots \xrightarrow{\bar{D}_{A}} A^{0,n} (  X, \mathfrak{u} (E)) 
  \longrightarrow 0 ,\\
\end{split}
\label{defHE}
\end{equation}
where 
$  A^{0,q} (X ,\mathfrak{u}(E)) 
:= C^{\infty} (\mathfrak{u} (E) \otimes A^{0,q} )$,  
$\mathfrak{u}(E) = \text{End} (E, h)$ is the bundle of 
skew-Hermitian endomorphisms of $E$, 
$A^{0,p}$ is the space of real 
$(0,p)$-forms (see \cite[pp.~32--33]{MR1004008}) 
over $X$, defined by 
$ A^{0,p} \otimes _{\R} \C =\Lambda^{0,p}  \oplus \Lambda^{p,0}$,  
\begin{equation*}
 \begin{split}
 \Omega^{+} (X ,\mathfrak{u}(E)) 
&:= A^{0,2} (X ,\mathfrak{u}(E)) 
  \oplus \Omega^0 (X ,\mathfrak{u}(E)) \omega \\
&= \{ \phi + \bar{\phi} + f \omega \, : \, 
\phi \in \Omega^{0,2} (X ,\mathfrak{u}(E)) ,\, f \in \Omega^{0} (X ,\mathfrak{u}(E))  \} ,
 \end{split}
\end{equation*} 
$\bar{D}_{A} : A^{0,p} (X , \mathfrak{u} (E)) \to  A^{0, p+1} (X , 
\mathfrak{u}(E))$ is defined by $\bar{D}_{A} \alpha = \bar{\partial}_{A} 
\alpha^{0,p} + \partial_{A} \overline{\alpha^{0,p}}$ for $\alpha = \alpha^{0,p} +
\overline{\alpha^{0,p}}$, where $\alpha^{0,p} \in \Omega^{0,p} (X ,
\mathfrak{u} (E))$, and 
$d_{A}^{+} := \pi^{+} \circ d_{A}, 
\bar{D}_{A} ' := \bar{D}_{A} 
\circ 
\pi^{0,2}$,  
where $\pi^{+} , \pi^{0,2}$ are respectively the orthogonal 
projections from $\Omega^2 $ to 
$\Omega^{+} , A^{0,2}$.

Kim proved that \eqref{defHE} is an elliptic complex if $A$ is a
Hermitian--Einstein connection. 
However, it is obviously not the Atiyah--Hitchin--Singer
type complex \cite{MR506229} if $n \geq 3$, since there are additional
terms such as   
$A^{0,3} ( X, \mathfrak{u} (E))$ and so on. 
Hence, the Hermitian--Einstein connections would not work for 
an analytic construction of the Donaldson--Thomas invariant just as it is. 
But, in \cite{Thomas00}, 
Thomas noted a perturbed Hermitian--Einstein equation, 
which basically corresponds to a ``holding'' of the extra 
term $A^{0,3} ( X, \mathfrak{u} (E))$ 
in \eqref{defHE} (we shall see it in Section 3.1), could
possibly work for an analytic definition of the Donaldson--Thomas invariant. 
We introduce that perturbed equation in the next subsection.

\subsection{The Donaldson--Thomas instantons on compact symplectic 6-manifolds}

Let $Z$ be a compact symplectic 6-manifold with 
 symplectic form $\omega$, and 
$E$ a unitary vector bundle of rank $r$ over $Z$.  
We take an almost complex structure on $Z$ compatible with the
 symplectic form $\omega$. 
Then the almost complex structure induces the 
 splitting of the complexified two forms as 
$ \Lambda^2 \otimes \C = \Lambda^{2,0} \oplus 
\Lambda^{0,2} \oplus \Lambda^{1,1}$.  
We consider the following equations for a connection $A$ of $E$, which
 preserves the hermitian structure of $E$, 
 and a $\mathfrak{u} (E)$-valued (0,3)-form $u$ on $Z$. 
\begin{gather}
F_{A}^{0,2} + \bar{\partial}_{A}^{*} u =0, 
\label{DT1}
\\
F_{A}^{1 ,1} \wedge \omega^{2} + [ u , \bar{u}]   
+  2 \pi i \mu(E) Id_{E} \,  
 \omega^3 = 0  ,
\label{DT2}
\end{gather}
where $F_{A}^{0,2}$ and $F_{A}^{1,1}$ are the (0,2) and (1,1) components
 of the curvature $F_{A}$ of $A$, 
and $\mu (E) := \frac{1}{r} \int_{Z} c_{1} (E) \wedge \omega^2$. 
We call these equations \eqref{DT1}, \eqref{DT2} the \textit{Donaldson--Thomas
equations}, and a solution $(A, u)$ to these equations  
a \textit{Donaldson--Thomas
instanton} (\textit{D--T instanton} for short).

One may think of 
these equations as the 
 Hermitian--Einstein equations 
with a perturbation $u$. 
However, we think of $u$ as a Higgs field, namely, a new variable.  
One of advantages of bringing in the new field $u$ is 
that the Donaldson--Thomas equations form an elliptic system after fixing a
gauge transformation, despite the fact that the
Hermitian--Einstein equations 
on compact K\"{a}hler threefolds do not form it in the same way.

These equations \eqref{DT1}, \eqref{DT2} were also studied in physics 
such as in \cite{BaKaSi98}. 
In that context, these equations are interpreted as a bosonic part of 
dimensional reduction equations of the $N=1$ super Yang--Mills equation in
10 dimensions to 6 dimensions (see also \cite{IqNeOkVa03}, \cite{NeOoVa04}).

\paragraph{The equations in the K\"{a}hler case.}

If the almost complex structure is integrable, then we have 
$\bar{\partial}_{A} F_{A}^{0,2} = 0$ by the Bianchi identity. 
Hence $\bar{\partial}_{A} \bar{\partial}_{A}^{*} u = 0 $ by \eqref{DT1},
thus we have $\bar{\partial}_{A}^{*} u = 0$ on compact K\"{a}hler threefolds.  
Therefore, the Donaldson--Thomas equations \eqref{DT1}, \eqref{DT2} becomes 
 \begin{gather*}
\bar{\partial}_{A}^{*} u = 0, \quad 
F_{A}^{0,2} = 0  , 
\label{DT1_k} \\
F_{A}^{1 ,1} \wedge \omega^{2} + [ u , \bar{u}]  
+  2 \pi i \mu(E)  Id_{E} \, \omega^3 =0. 
\label{DT2_k}
\end{gather*}
The above equations could be thought of as a generalisation of the
Hitchin equation on Riemann surfaces \cite{Hit} to K\"{a}hler threefolds
in the same way as the Vafa--Witten equations on K\"{a}hler surfaces as
mentioned in \cite{Tan}. 
In Section \ref{sec:HKDT} to this article, we describe the corresponding Hitchin--Kobayashi correspondence in this
setting, which turns out to be a special case of results by 
\'{A}lvarez-C\'{o}nsul and Garc\'{\i}a-Prada \cite{AG}.

\subsection{The complex ASD and the Donaldson--Thomas instantons}

In this section, we see that the Donaldson--Thomas
equations on Calabi--Yau threefolds 
can be thought of as the dimensional reduction of 
the complex ASD equations on Calabi--Yau fourfolds, 
this was pointed out by Tian \cite{Tian00}, and 
it is analogous to the Hitchin pair \cite{Hit}.

\paragraph{Complex ASD equations on Calabi--Yau fourfolds.}

Let $X$ be a compact Calabi--Yau fourfold with K\"{a}hler form
$\omega$ 
and holomorphic $(4,0)$-form $\theta$. 
We assume the normalization condition  
$ \theta \wedge \bar{\theta} = \frac{16}{4!} \omega^4$  on $\omega$ and
$\theta$.
Let $E$ be a hermitian vector bundle over $X$.  
By using the holomorphic $(4,0)$-form $\theta$, 
we define {\it the complex Hodge operator} 
$ *_{\theta} : \Lambda^{0,2} \to \Lambda^{0,2} $ 
by 
$ \text{tr} ( \phi  \wedge *_{\theta} \psi )
 = \langle \phi , \psi \rangle  \bar{\theta}$ for $ \phi , \psi \in
 \Lambda^{0,2} $. 
Then $*_{\theta}^2 =1$, and the space of 
$(0,2)$-forms further decomposes into  
$ \Lambda^{0,2} = \Lambda^{0,2}_{+} \oplus \Lambda_{-}^{0,2}$,  
where 
$ \Lambda^{0,2}_{+} 
= \{ \phi \in \Lambda^{0,2} \, : \, *_{\theta} \phi = \phi \} , \, 
 \Lambda^{0,2}_{-} 
= \{ \phi \in \Lambda^{0,2} \, : \, *_{\theta} \phi = - \phi \}$. 
Note that the operator $*_{\theta}$ is an anti-holomorphic map, hence 
$\Lambda^{0,2}_{+}$ and $\Lambda^{0,2}_{-}$ are real subspaces of 
$\Lambda^{0,2}$.

We consider the following equations for connections of $E$: 
\begin{gather}
(1 + *_{\theta}) F_{A}^{0,2} = 0 , 
\quad i \Lambda F_{A}^{1,1}  = 8 \pi \mu (E) Id_{E} ,  
\label{casd}
\end{gather}
where $\mu (E) := \frac{1}{r} \int_{X} c_{1} (E) \wedge \omega^3$. 
We call these equations \textit{complex
ASD equations}, and a solution to these equations a {\it complex ASD
instanton}. 
These were brought in by Donaldson and  Thomas in \cite{DoTh98}. 
These equations with a gauge fixing equation form an elliptic system. 
Analytic properties of the complex ASD instantons were studied by Tian 
\cite{Tian00}.

Note that the complex ASD instantons are special cases of
$Spin(7)$-instantons on $Spin(7)$-manifolds (see \cite[\S~3.1]{Tan2}).

More recently, Donaldson--Thomas style invariants for Calabi--Yau
fourfolds, which concerns the moduli space of the solutions to the
above complex ASD equations, were defined by Borisov--Joyce \cite{BJ}, 
Cao \cite{C} and Cao--Leung
\cite{CL1} (see also \cite{CL2}, \cite{CL3}, \cite{CL4}).

\paragraph{Dimensional reduction.}

We describe a relation between 
the Donaldson--Thomas equations \eqref{DT1}, \eqref{DT2} 
and the complex ASD equations \eqref{casd} 
by dimensional reduction argument.   
This was pointed out by Tian \cite{Tian00}.

Let $Z$ be a compact Calabi--Yau threefold with K\"{a}hler form
$\omega_0$ and holomorphic $(3,0)$-form $\theta_0$, 
and $T^2$ a torus of complex dimension one. 
We consider the direct product of $Z$ and $T^2$, and denote it  by 
$X$, namely, $X := Z \times T^2$. 
We define a K\"{a}hler form $\omega$ and a 
holomorphic $(4,0)$-form on $X$ 
by 
$\omega := \omega_0 + dz \wedge d \bar{z} , \quad \theta := \theta_0
\wedge dz$, 
where $dz$ is the standard flat $(1 ,0)$ form on $T^2$.

Let $E$ be a hermitian vector bundle with structure group 
$SU(r)$ 
over $Z$, and $p : X = Z \times T^2 \to Z$. 
We then consider $T^2$-invariant solutions to the complex ASD 
equations \eqref{casd} on  $p^{*} (E) \to X$. 
Then these solutions satisfy the Donaldson--Thomas equations on $Z$. 
In fact, if we write a connection $A$ on $ X = Z \times T^2$ 
as 
$ A_{X} = A + \phi dz + \bar{\phi} d \bar{z}$, 
where  $A$ is the $Z$-component of the connection $A_X$ 
and $\phi  \in \Gamma ( Z , \mathfrak{su} (E))$, 
then the curvature becomes 
$$F_{A_X} = F_{A} + d_{A} \phi \wedge dz 
+ d_{A} \bar{\phi} \wedge d \bar{z} +[ \phi , \bar{\phi} ] dz \wedge
d\bar{z} . $$
Hence, if we put  
$u := \phi \, \bar{\theta}_0 \in \Omega^{0,3} (Z , \mathfrak{su} (E))$, 
then $A$ and $u$ satisfy
 the Donaldson--Thomas equations, provided that this $A_X$  is a $T^2$-invariant
solution to the complex ASD equations.

\section{Local model for  
the moduli space of Donaldson--Thomas instantons}
\label{sec:KDT}

Let $Z$ be a compact symplectic 6-manifold with symplectic 
form $\omega$, 
$(E,h)$ a hermitian vector bundle over $Z$ with hermitian metric $h$.

We denote by 
${\mathcal A} ( E ) = {\mathcal A} ( E ,h )$ 
the set of all connections of $E$ 
which preserve the hermitian structure of $E$, and 
put $ {\mathcal C} (E ) := 
{\mathcal A} (E ) \times \Omega^{0,3} 
(Z, \mathfrak{u} (E))$. 
We denote by ${\mathcal G} (E) = \mathcal{G} (E, h)$ 
the gauge group, the group of unitary automorphism of $(E,h)$, 
where the action of
the gauge group on $\mathcal{C}(E)$ is defined by 
$g(A,u) = ( A - (d_{A} g) g^{-1},  g^{-1} u g  )$. 
These spaces $\mathcal{C} (E), \, \mathcal{G} (E)$ 
can be seen as Fr\'{e}chet spaces with $C^{\infty}$-norms, 
but we shall use Sobolev completions of them in Section 3.2.

We denote by $\Gamma_{(A,u)}$ the stabilizer at $(A, u) \in
\mathcal{C}(E)$ of the gauge group $\mathcal{G}(E)$, namely, 
$ \Gamma_{(A,u)} := 
\{ g \in \mathcal{G} (E) \, : \, g (A, u) = (A ,u)\} $. 
We call $(A,u) \in \mathcal{C}(E)$ 
{\it irreducible} 
if  $\Gamma_{(A,u)}$
coincides with the centre of the structure group of $E$, and
{\it reducible} otherwise.  
We denote by $\mathcal{C}^{*} (E)$ the set of all irreducible pair $(A,u) \in
\mathcal{C}(E)$. 
Note that the action of $\mathcal{G} (E)$ is not free on 
$\mathcal{C}^{*}(E)$, but the action of 
 $\hat{\mathcal{G}} (E) = \mathcal{G} (E) / U(1)$ is free on $\mathcal{C}^{*}(E)$.

We denote by $ \mathcal{D} (E)$ the set of all D--T 
instantons of $E$, and by $ \mathcal{D}^{*} (E)$ the set of all
 irreducible D--T 
instantons of $E$.  
We call 
$\mathcal{M} (E) 
= \mathcal{D} (E) / \mathcal{G} ( E ) $
the {\it moduli space of the Donaldson--Thomas instantons}.

\subsection{Linearization}

The infinitesimal deformation of a D--T 
instanton $(A, u)$ is described by the following sequence: 
\begin{equation}
\begin{split}
  0 \longrightarrow \Omega^{0} (Z, \mathfrak{u} (E))
 \xrightarrow{\, \, \, \, \, D_{(A,u)} \, \, \, \, \,} 
 \Omega^{1} (Z, &\mathfrak{u} (E)) 
  \oplus A^{0,3} (Z, \mathfrak{u} (E)) \\
 &\xrightarrow{\, \, \, \, \, D_{(A,u)}^{+} \, \, \, \, \, } 
 \Omega^{+} (Z, \mathfrak{u} (E)) \longrightarrow 0  , 
\end{split}
\label{DTcomplex}
\end{equation}
where 
$D_{(A,u)} (s) =
( d_{A} s , [ \tilde{u}, s] ) 
, \,  
\tilde{u} = u + \bar{u}
, \, 
 D_{(A,u)}^{+}  (  \alpha ,\upsilon ) =
 d_{A}^{+} \alpha  +  \Lambda^2 ( [ u, \bar{\upsilon} ] + [ \upsilon , \bar{u}
 ] ) 
+ \bar{D}_{A}^{*} \upsilon $   
for $s \in \Omega^{0}  (Z, \mathfrak{u} (E))$ and 
$(\alpha  ,\upsilon ) \in  \Omega^{1} (Z, \mathfrak{u} (E)) 
\oplus A^{0,3} 
(Z, \mathfrak{u} (E))$.  
If $(A, u)$ is a D--T instanton, then \eqref{DTcomplex} is
a 
complex. 
In fact, $D_{(A,u)}^{+}  D_{(A,u)} = 0$ follows directly 
from the equations \eqref{DT1}, \eqref{DT2}. 
The complex \eqref{DTcomplex} can be seen as ``holding'' of 
the $A^{0,3} (Z , \mathfrak{u}(E))$-term in \eqref{defHE}, 
namely, it is equivalent to consider the following complex
instead of \eqref{DTcomplex}. 
\begin{equation}
\begin{split}
  0 \longrightarrow \Omega^{0} ( X, \mathfrak{u}(E)) 
  &\xrightarrow{d_{A}} 
 \Omega^1 ( X, \mathfrak{u}(E)) \\ 
 &\xrightarrow{d_{A}^{+}} 
 \Omega^{+} (X,  \mathfrak{u}(E)) 
   \xrightarrow{\bar{D}_{A}'} 
A^{0,3} ( X, \mathfrak{u} (E)) \longrightarrow 0. \\ 
\end{split}
\label{alcomp}
\end{equation}
This is the same as that of the Hermitian--Einstein connections in Section 2.2, 
but it still makes sense in the almost complex setting. 
Hence the following just reduces to the case in \eqref{alcomp}, and it
was proved by Reyes Carri\'{o}n \cite{Reyes_Carrion98}.

\begin{proposition}
If $(A,u) \in \mathcal{D} (E)$, then 
the complex \eqref{DTcomplex} is elliptic. 
\end{proposition}

We denote by $H^{i}_{(A,u)} = H^{i}_{(A,u)} (Z, \mathfrak{u} (E))$ the 
$i$-th cohomology of the complex \eqref{DTcomplex} for $i = 0,1,2$.

The complex \eqref{alcomp} has the associated Dolbeault complex as Kim
\cite{MR888136} described it in the K\"{a}hler case 
(see also \cite[Chap.~V$\!$I$\!$I \S2]{MR909698}):  
\begin{equation}
\begin{CD}
0 @>>> \Omega^0 @>d_{A}>> \Omega^1 @>d_{A}^{+}>>\Omega^{+} @>\bar{D}_{A}'>>  
A^{0,3} @>\bar{D}_{A}>>  0 \\ 
@. @VVj_0V @VVj_1V @VVj_2V @VVj_3V  \\ 
0 @>>> \Omega^{0,0} @>\bar{\partial}_{A}>> \Omega^{0,1} 
 @>\bar{\partial}_{A}>>  
\Omega^{0,2}  
@>\bar{\partial}_{A}>> \Omega^{0,3} @>\bar{\partial}_{A}>>  0 , 
\end{CD} 
\label{comm}
\end{equation}
where $j_{0}$ is injective, $j_1$ is bijective, $j_2$ is surjective with  
the kernel $\{ \beta \omega \, : \,  \beta \in \Omega^{0} \}$, and 
$j_{3}$ is bijective.  
Hence the index of the complex \eqref{alcomp}, 
thus that of the complex \eqref{DTcomplex}, can be expressed by that of the
Dolbeault complex above, which is given by 
$
 \int_{Z} \hat{A} (Z) \wedge ch (K_{Z}^{\frac{1}{2}}) 
\wedge ch( \mathfrak{u} (E))  
$ (See \cite[\S3.5]{G}). 
In the K\"{a}hler case,  the index can be computed as 
$$   \int_{Z} c_{1} ( Z 
)\wedge 
\left(  \frac{r -1 }{2} c_{1} (E)^2 -  r c_{2} (E) \right) 
 + r^2 \sum_{i =0}^{3} (-1)^{i} \dim H^{0,i} (Z). $$ 
Note that the index is zero if $Z$ is a Calabi--Yau threefold.

\subsection{Kuranishi model and the local description of the moduli space}

We denote by $\mathcal{C}_{k} (E) ,  
\mathcal{C}^{*}_{k}(E) , \mathcal{D}_{k} (E) , 
\mathcal{D}_{k}^{*} (E)$ 
the 
 $L^{2}_{k}$-completions of $\mathcal{C} (E)$, $  
\mathcal{C}^{*} (E)$, $\mathcal{D} (E)$, $\mathcal{D}^{*}(E)$
respectively, and  
by $\mathcal{G}_{k+1} (E)$ 
the $L^{2}_{k+1}$-completion of $\mathcal{G} (E)$. 
We take $k$ sufficiently large so that $\mathcal{G}_{k+1}$ becomes a Hilbert Lie group
acting smoothly on $\mathcal{C}_{k} (E)$, 
the quotient topology 
$\mathcal{C}_{k} (E) / \mathcal{G}_{k+1} (E)$ becomes Hausdorff 
 (see e.g. \cite[\S 3]{FrUh91}), 
and to use implicit function theorems for the Sobolev spaces. 
A general reference for the Sobolev spaces and the implicit function
theorems on them for our purpose is, for example, \cite{W}.

\paragraph{Slice.}

We define  {\it slice} $S_{(A,u), \varepsilon}$  
at $(A,u)$ in $\mathcal{C}_{k} (E)$ by 
\begin{equation*}
\begin{split}
&S_{(A,u) , \varepsilon} \\
&:= \{ (\alpha , \upsilon ) 
\in L^{2}_{k} \left( \mathfrak{u} (E) \otimes (\Lambda^1 
\oplus A^{0,3})  \right) :  D_{(A,u)}^{*} (\alpha , 
\upsilon) =0 , \, || (\alpha , \upsilon) ||_{L^{2}_{k}} \leq \varepsilon
\}. 
\end{split}
\end{equation*} 
This set $S_{(A,u) , \varepsilon}$ is transverse to the
$\mathcal{G}_{k+1}$-orbit through $(A,u)$ as $\ker D_{(A,u)}^{*}$ is
orthogonal to $\text{Im}\,  D_{(A,u)}$ with respect to the $L^2$-norm in 
$L^{2}_{k} \left( \mathfrak{u} (E) \otimes (\Lambda^1 
\oplus A^{0,3})  \right)$. 
There is a natural map 
$ P_{(A, u) , \varepsilon}  
: S_{(A,u), \varepsilon} 
\to \mathcal{C}_{k}(E) /
\mathcal{G}_{k+1} (E) $ defined by 
$(\alpha, \upsilon) \mapsto [(A + \alpha ,
u + \upsilon')] $, where $\upsilon' = j_{3} (\upsilon)$, and 
$j_3 : A^{0,3} \to \Omega^{0,3}$ is the map in \eqref{comm}.

In the following, we take $(A,u) \in \mathcal{C}^{*}_{k} (E)$ for
simplicity.

\begin{proposition}
Let $(A,u) \in \mathcal{C}^{*}_{k}(E)$. 
Then there exists $\varepsilon >0$ such that 
$S_{(A,u) , \varepsilon}$ is diffeomorphic to 
$ P_{(A, u) , \varepsilon}  
\left( S_{(A,u) , \varepsilon} \right)$ in 
$\mathcal{C}^{*}_{k} (E)/ \hat{\mathcal{G}}_{k+1} (E)$. 
\end{proposition}

\begin{proof}
This is a familiar claim in gauge theory, 
the proof is a modification of known results for the ASD and the 
 Hermitian--Einstein connections (cf. \cite[Th.~6]{D},  
\cite[Th.~3.2, Th.~4.4]{FrUh91}, 
\cite[Chap.~V$\!$I$\!$I \S4 Th.~4.16]{MR909698}, 
and \cite[Prop.~4.2.1]{LT}). 
We divide the proof into two steps:

\vspace{0.1cm}
\underline{Step 1.} 
\hspace{0.1cm} We consider a map 
$ f_{(A,u)} : 
S_{(A,u) ,\varepsilon} \times 
\hat{\mathcal{G}}_{k+1} (E) \to \mathcal{C}_{k}^{*} (E) $ 
defined by $f_{(A,u)} ( ( \alpha , \upsilon ) ,g) 
= g (A + \alpha , u + \upsilon')$. 
Then the differential of $f_{(A,u)}$ at $( (0,0) , id )$ is given by  
$ D f_{(A,u)} |_{((0,0), id)} ( (\beta , \varphi) , s ) 
= ( \beta , \varphi ) + D_{(A,u)} (s)$.  
As $\text{Im}\, D_{(A,u)}$ and $\ker D_{(A,u)}^{*}$ are
 $L^2$-orthogonal in $L^{2}_{k} \left( \mathfrak{u} (E) \otimes (\Lambda^1 
\oplus A^{0,3})  \right)$, 
$ D f_{(A,u)} |_{((0,0), id)}$ is injective if $(A,u)$ is
 irreducible.

On the other hand, associated to 
the operator 
$$D_{(A,u)}^{*} D_{(A,u)} : 
L^{2}_{k+1} 
( \mathfrak{u} (E) \otimes \Lambda^{0}) / \mathfrak{u} (1) \to 
 L^{2}_{k-1} 
( \mathfrak{u} (E) \otimes \Lambda^{0}) /  \mathfrak{u} (1), $$ 
where $L^{2}_{k+1} 
( \mathfrak{u} (E) \otimes \Lambda^{0}) / \mathfrak{u} (1) 
= \{ s \in L^{2}_{k+1} 
( \mathfrak{u} (E) \otimes \Lambda^{0}) \, : \, 
\int_{Z} \text{tr}\, (s) \, \text{vol}_{g} = 0 \}$,  
there exist the Green operator 
$G^0 : L^{2}_{k} ( \mathfrak{u}(E) \otimes \Lambda^{0}) 
 / \mathfrak{u} (1) 
\to L^{2}_{k} ( \mathfrak{u}(E) \otimes \Lambda^{0}) 
 / \mathfrak{u} (1) $ and 
the harmonic projection 
$H^{0}: L^{2}_{k} ( \mathfrak{u}(E) \otimes \Lambda^{0}) 
/ \mathfrak{u} (1)  
\to L^{2}_{k} ( \mathfrak{u}(E) \otimes \Lambda^{0}) 
/ \mathfrak{u} (1) $ 
with the identity: 
$$ Id = H^{0}  + D_{(A,u)}^{*} D_{(A,u)}  \circ G^{0} $$ 
(see e.g. \cite[Chap.~I$\!$V \S5]{W}).
From the identity, 
we obtain 
$D_{(A,u)}^{*} ( (\gamma , \chi ) - D_{(A,u)} G^{0} D_{(A,u)}^{*} 
(\gamma , \chi ) ) = 0$ 
for any $(\gamma , \chi ) \in L^{2}_{k} 
(\mathfrak{u} (E) \otimes (\Lambda^{1} \oplus A^{0,3}))$. 
Thus, for a given $(\gamma , \chi ) \in L^{2}_{k} 
(\mathfrak{u} (E) \otimes (\Lambda^{1} \oplus A^{0,3}))$, 
we take $(\beta, \varphi) 
= (\gamma ,\chi) -  D_{(A,u)} G^{0} D_{(A,u)}^{*} 
(\gamma , \chi ), \, s = G^{0} D_{(A,u)}^{*} 
(\gamma , \chi )$ to get $(\gamma , \chi) 
= (\beta , \varphi ) + D_{(A,u)} (s)$. 
Therefore $D f_{(A,u)} |_{((0,0), id)}$ is surjective.

We then use an inverse mapping theorem 
for the Hilbert spaces (see e.g. \cite[Chap.~6]{Lang}) 
to deduce that around $(A,u)$,   
$\mathcal{C}_{k}^{*} (E)$ is locally diffeomorphic to 
a neighbourhood of $((A ,u) , id )$ in $S_{(A,u), \varepsilon} \times 
\hat{\mathcal{G}}_{k+1} (E)$.

\vspace{0.1cm}
\underline{Step 2.} 
\hspace{0.1cm} 
We then prove that if for $(\alpha_1 , \upsilon_1) , (\alpha_2 , \upsilon_2) 
\in S_{(A,u) , \varepsilon}$ there exists 
$g \in \mathcal{G}_{k+1} (E)$ such that 
\begin{equation}
 (A + \alpha_1 , \tilde{u} + \upsilon_1) 
= g (A + \alpha_2 , \tilde{u} + \upsilon_2),  
\label{ga}
\end{equation}
then $c g$ is close to $id_{E}$ in $L^{2}_{k+1}$ 
for some $c \in U(1)$.

Since we assume that $(A,u)$ is irreducible, we can take 
$c \in U(1)$ so that 
$g' = cg -id_{E} \in \ker \left( D_{(A,u)}\right)^{\perp}$.  
From \eqref{ga}, we get 
$d_{A} g' = \alpha_1 g' - g' \alpha_2 + (\alpha_1 -\alpha_2) ,  
\, [\tilde{u} , g'] = g' \upsilon_1 - \upsilon_2 g' + \upsilon_1 - \upsilon_2 $. 
Hence, 
\begin{equation} 
D_{(A,u)} g' = 
(\alpha_1 g' - g' \alpha_2  + \alpha_{12} , 
g' \upsilon_1 - \upsilon_2 g' + \upsilon_{12})
 , 
\label{dgp}
\end{equation}
where $\alpha_{12} = \alpha_1 -\alpha _2 , \upsilon_{12} = \upsilon_1 - \upsilon_2$.

Since $g'$ lies in $\left( 
\ker D_{(A,u)} \right)^{\perp}$, 
there exists a constant $C >0$ independent of $(A ,u)$ and $g'$
 such that 
$ 
|| g' ||_{L^{2}_{k+1}} 
\leq C || D_{(A,u)} g' ||_{L^{2}_{k}}
$. 
Thus, using \eqref{dgp}, we obtain 
$$ || g' ||_{L^{2}_{k+1}} 
\leq C \left( || g' ||_{L^{2}_{k}} 
 \left( || \alpha_{1} ||_{L^{2}_{k}} 
+ || \alpha_{2} ||_{L^{2}_{k}}  
+ || \upsilon_{2} ||_{L^{2}_{k}} 
\right) 
 + || \alpha_{12} ||_{L^{2}_{k}}  
 + || \upsilon_{12}  ||_{L^{2}_{k}} \right) . $$
Hence, 
$$ || g' ||_{L^{2}_{k+1}} 
\leq \frac{C}
{1 - 3 \varepsilon C} 
\left( || \alpha_{12} ||_{L^{2}_{k}}  
+ || \upsilon_{12} ||_{L^{2}_{k}} \right) $$
for $\varepsilon < 1 /3 C$. 
Thus, we get  
$|| c g - id_{E} ||_{L^{2}_{k+1}} < C' \varepsilon$ for $\varepsilon$
 small, 
where $C'$ is a positive
 constant.

From this, the assertion of the lemma is reduced to Step 1.

\end{proof}

\begin{remark}
By modifying the proof of Lemma 3.2,  
one can prove that for $(A,u) \in \mathcal{C}_{k} (E)$, there exists
 $\varepsilon >0$ such that $S_{(A,u), \varepsilon} / \hat{\Gamma}_{(A,u)}$ is
 diffeomorphic to $P_{(A,u)} \left( S_{(A,u), \varepsilon} / \hat{\Gamma}_{(A,u)} 
\right)$ in $\mathcal{C}_{k} (E) / 
\hat{\mathcal{G}}_{k+1} (E)$, 
where $\hat{\Gamma}_{(A,u)} = \Gamma_{(A,u)}/ U(1)$, 
following, for example, \cite[Th.~4.4]{FrUh91}.  
\label{resl}
\end{remark}

\paragraph{Kuranishi model.}

This is also a familiar picture in gauge theory. 
We describe it for the Donaldson--Thomas instanton case, 
modifying known results in the ASD and Hermitian--Einstein 
connections (cf. \cite[Prop.~8]{D},  
\cite[Chap.V$\!$I$\!$I \S 4 Th.~4.20]{MR909698}, 
and \cite[Prop.~4.5.3]{LT}). 
We take $(A,u) \in \mathcal{D}_{k} (E)$, and consider a deformation 
$(A + \alpha  , u + \upsilon') \in \mathcal{D}_{k} (E)$, 
where $ (\alpha , \upsilon) \in  L^{2}_{k} (\mathfrak{u} (E) 
\otimes (\Lambda^{1} 
\oplus  A^{0,3} ))$. 
Then, $(\alpha , \upsilon)$ satisfies the following:  
\begin{gather}
 d_{A}^{+} \alpha  + \pi^{+} 
(\alpha \wedge \alpha )
+ B_{u} ( \upsilon )  
+ \Lambda^2 [ \upsilon , \bar{\upsilon} ] +  
\bar{D}_{A}^{*} \upsilon 
+ \bar{*} \alpha \bar{*} \upsilon = 0 , 
\label{deform_a} 
\end{gather}
where $ B_{u} ( \upsilon ) := \Lambda^2 (  
 [ u  , \bar{\upsilon}] + [ \upsilon , \bar{u} ] )$.

Associated to the operator $$ D_{(A,u)}^{+} (D_{(A,u)}^{+})^{*} 
: L^{2}_{k} ( \mathfrak{u}(E) \otimes \Lambda^{+}) 
\to L^{2}_{k} ( \mathfrak{u}(E) \otimes \Lambda^{+}), $$
there exist the Green operator 
$G^2 : L^{2}_{k} ( \mathfrak{u}(E) \otimes \Lambda^{+}) 
\to L^{2}_{k} ( \mathfrak{u}(E) \otimes \Lambda^{+})$ and 
the harmonic projection 
$H : L^{2}_{k} ( \mathfrak{u}(E) \otimes \Lambda^{+}) 
\to L^{2}_{k} ( \mathfrak{u}(E) \otimes \Lambda^{+})$ 
with the identity: 
$$Id = H  + D_{(A,u)}^{+} (D_{(A,u)}^{+})^{*} \circ G^{2} $$
(see e.g. \cite[Chap.I$\!$V \S 5]{W}).
 Using these, we define a map 
$$K_{(A,u)} :  
L^{2}_{k} ( \mathfrak{u}(E) \otimes( \Lambda^{1} 
\oplus A^{0,3} ))
\to 
L^{2}_{k} ( \mathfrak{u}(E) \otimes( \Lambda^{1} 
\oplus A^{0,3} ))$$
by 
$K_{(A,u)} 
( \alpha ,\upsilon ) := 
(\alpha +  (d_{A}^{+})^{*} \circ G^{2} \circ 
( \pi^{+} ( \alpha \wedge \alpha)   + \Lambda^2 [\upsilon , \bar{\upsilon}] 
+ \bar{*} \alpha \bar{*} \upsilon )  , \,  
 \upsilon + ( \bar{D}_{A}' + (B_{u}^{*})' ) \circ G^{2} \circ 
( \pi^{+} ( \alpha \wedge \alpha)  + \Lambda^2 [ \upsilon ,
 \bar{\upsilon} ] 
+ \bar{*} \alpha \bar{*} \upsilon ) 
)$, 
where 
$( B_{u}^{*})' = B_{u}^{*} \circ \pi^{\omega}$, 
$B_{u}^{*} : \Omega^{0}\omega \to A^{0,3}$ is the adjoint of $B_{u}$,
and 
$\pi^{\omega}$ is the orthogonal projection from $\Omega^2$ to $\Omega^0
\omega$.

\begin{lemma}
A pair $(\alpha , \upsilon) \in 
L^{2}_{k} (\mathfrak{u} (E) \otimes (\Lambda^{1} 
\oplus  A^{0,3} ))$ satisfies \eqref{deform_a} 
if and only if it satisfies $
D_{(A,u)}^{+} K_{(A,u)}( \alpha , \upsilon ) = 0 $ 
and $
 H  ( \pi^{+} ( \alpha \wedge \alpha ) + \Lambda^2 
[\upsilon ,\bar{\upsilon}] + \bar{*} \alpha \bar{*}\upsilon )
  =0$. 
\label{lem}
\end{lemma}

\begin{proof}
Using the identity $Id = H + D_{(A,u)}^{+} (D_{(A,u)}^{+})^{*} 
\circ G^{2}$, we rewrite 
the left-hand side of \eqref{deform_a} as  
\begin{equation}
\begin{split}
& d_{A}^{+} ( \alpha +  ( d_{A}^{+})^{*}  \circ G^{2} \circ 
( \pi^{+} ( \alpha \wedge \alpha ) + \Lambda^2 [\upsilon ,
 \bar{\upsilon}] 
+ \bar{*} \alpha \bar{*} \upsilon  ) )  
+ B_{u} ( \upsilon )  
\\
 & \quad  +  \bar{D}_{A}^{*} \left( \upsilon + 
( \bar{D}_{A}' + (B_{u}^{*})' )   \circ G^{2} \circ 
( \pi^{+} ( \alpha \wedge \alpha )  + \Lambda^2 [ \upsilon ,
 \bar{\upsilon} ] + \bar{*} \alpha \bar{*} \upsilon  
)) \right)
 \\
 &  \quad \quad 
 + H \circ ( \pi^{+} ( \alpha \wedge \alpha ) +  \Lambda^2 [\upsilon ,
 \bar{\upsilon} ] 
+ \bar{*} \alpha \bar{*} \upsilon  ))  \\
&  \quad  \quad \quad =  
D_{(A,u)}^{+} K_{(A,u)} + 
H \circ ( \pi^{+} ( \alpha \wedge \alpha ) +  \Lambda^2 [\upsilon ,
 \bar{\upsilon} ] 
+ \bar{*} \alpha \bar{*}  \upsilon  )) . 
\end{split}
\label{dcp}
\end{equation} 
Hence, if $D_{(A,u)}^{+} K_{(A,u)} =0$ and $H \circ ( \pi^{+} ( \alpha \wedge \alpha ) +  \Lambda^2 [\upsilon ,
 \bar{\upsilon} ] 
+ \bar{*} \alpha \bar{*} \upsilon  )) =0$, 
then \eqref{deform_a} holds.

Conversely, if \eqref{deform_a} holds, then from \eqref{dcp} we get 
$$D_{(A,u)}^{+} K_{(A,u)} + 
H \circ ( \pi^{+} ( \alpha \wedge \alpha ) +  \Lambda^2 [\upsilon ,
 \bar{\upsilon} ] 
+ \bar{*} \alpha \bar{*} \upsilon  )) =0 . $$ 
Thus, $( D_{(A,u)}^{+})^{*} D_{(A,u)}^{+} K_{(A,u)} =0 $. 
This implies 
$|| D_{(A,u)}^{+} K_{(A,u)} ||_{L^{2}_{k-1} (\mathfrak{u} (E) 
\otimes \Lambda^{+})} = 0$, 
hence, 
$D_{(A,u)}^{+} K_{(A,u)} = 0$ and $H \circ ( \pi^{+} ( \alpha \wedge \alpha ) +  \Lambda^2 [\upsilon ,
 \bar{\upsilon} ] 
+ \bar{*} \alpha \bar{*} \upsilon  )) =0$. 
\end{proof}

We put $S_{(A,u), \varepsilon}^{d} 
:= \{ ( \alpha , \upsilon ) \in S_{(A,u) , \varepsilon} \, 
 : \, (\alpha , \upsilon ) \text{ satisfies \eqref{deform_a}} \}$, and 
denote by 
$\bold{H}^{i}_{(A,u)} (Z , \mathfrak{u}(E)) \, ( i = 0,1,2)$ 
the harmonic spaces of the
complex \eqref{DTcomplex}.

\begin{lemma}
  $$K_{(A,u)} (S_{(A,u) , \varepsilon}^{d}) \subset \bold{H}^{1}_{(A,u)} 
(Z, \mathfrak{u}(E)). $$
\label{sha}
\end{lemma}

\begin{proof}

From the definition of the map $K_{(A,u)}$, 
we have 
\begin{equation*}
\begin{split}
&D_{(A,,u)}^{*} 
K_{(A,,u)} (\alpha , \upsilon) \\
& \qquad = D_{(A,u)}^{*} (\alpha , \upsilon) 
 +  D_{(A,u)}^{*} ( D_{(A,u)}^{+})^{*} 
(G^{2} \circ 
( \pi^{+} ( \alpha \wedge \alpha ) + \Lambda^2 [\upsilon ,
 \bar{\upsilon}] 
+ \bar{*} \alpha \bar{*}  \upsilon  ))) \\
\end{split}
\end{equation*}
for $(\alpha , \upsilon) \in S_{(A,u) , \varepsilon}^{d} 
 $. 
This is equal to $0$, because $D_{(A,u)}^{*} (\alpha , \upsilon) =0$ 
for $(\alpha , \upsilon) \in S_{(A,u) , \varepsilon}^{d}$, and 
$D_{(A,u)}^{*} (D_{(A,u)}^{+})^{*} = 0$ as 
$D_{(A,u)}^{+} D_{(A,u)} =0$. 
From Lemma 3.4, 
we also have $D_{(A,u)}^{+} K_{(A,u)} =0$. Thus Lemma 3.5 holds. 
\end{proof}

From Lemmas 3.4 and 3.5, we deduce the following.  
\begin{lemma}
A pair $(\alpha , \upsilon) \in 
L^{2}_{k} (\mathfrak{u} (E) \otimes (\Lambda^{1} 
\oplus  A^{0,3} ))$ lies in $S_{(A,u) , \varepsilon}^{d}$ if and only if 
$K_{(A,u)} (\alpha , \upsilon) \in \bold{H}^{1}_{(A,u)} 
(Z , \mathfrak{u}(E))$ and $H  ( \pi^{+} ( \alpha \wedge \alpha ) + \Lambda^2 [\upsilon 
, \bar{\upsilon}] + \bar{*} \alpha \bar{*} \upsilon )
  =0$. 
\end{lemma}

We now prove the following. 
\begin{theorem}
Let $(A,u) \in \mathcal{D}^{*} (E)$. Then  
there exists a neighbourhood $U$ of $0$ in $\bold{H}^{1}_{(A,u)} 
(Z , \mathfrak{u}(E) )$ 
such that around $[(A ,u)]$ the moduli space $\mathcal{M}^{*} (E) = 
\mathcal{D}^{*} (E) / \hat{\mathcal{G}} (E)$ is locally
 modeled on the zero set of a real analytic 
map 
$\kappa_{(A,u)} : U \to \bold{H}^{2}_{(A,u)} (Z ,\mathfrak{u}(E)) $ 
with $\kappa_{(A,u)} (0) =0$, and 
the first derivative of $\kappa_{(A ,u)}$ at $0$ also vanishes. 
\label{kur}
\end{theorem}

\begin{proof}
From the definition of 
the map $K_{(A,u)}$,  we have $K_{(A,u)} (0) = 0$.  
Since the differential of $K_{(A,u)}$ at $0$ is identity, 
we can deduce,  
from the inverse mapping theorem on the 
Hilbert spaces (see e.g. \cite[Chap.~6]{Lang}), 
that there exist
a neighbourhood $U$ of $0$ in $\bold{H}^{1}_{(A,u)} (Z ,\mathfrak{u}(E))$ 
and a map 
$K^{-1}_{(A,u)} : U \to L^{2}_{k} ( \mathfrak{u} (E) \otimes 
(\Lambda^{1} \oplus A^{0,3} ))$
such that 
$K^{-1}_{(A,u)}$ is a diffeomorphism between $U$ and $K^{-1}_{(A,u)} (U)$.  
We then define a map $\kappa_{(A,u)} : U \to \bold{H}^{2}_{(A,u)}$ by 
$\kappa_{(A,u)} = \psi \circ K_{(A,u)}^{-1}$, where 
$\psi : \bold{H}^{1}_{(A,u)} 
\to \bold{H}^{2}_{(A,u)}$ is defined by 
$\psi (\alpha , \upsilon) = H  
( \pi^{+} ( \alpha \wedge \alpha ) + \Lambda^2 [\upsilon 
, \bar{\upsilon}] 
+ \bar{*} \alpha \bar{*} \upsilon )$.

We now take $\varepsilon$ sufficiently small so that all the following
 hold.  
Firstly, from Lemma 3.6, 
the zero set of $\kappa_{(A,u)}$  is mapped by 
$K_{(A,u)}^{-1}$  diffeomorphically to an open subset in  
$S_{(A,u) , \varepsilon}^{d}$.  
Next, from Proposition 3.2, $S_{(A,u), \varepsilon}^{d}$ is diffeomorphic to 
$p_{(A,u), \varepsilon} ( S_{(A,u) , \varepsilon}^{d} )$ 
in $\mathcal{D}^{*}_{k} (E) / \hat{\mathcal{G}}_{k+1} (E)$. 
Hence,  
the zero set of $\kappa_{(A,u)}$ 
is diffeomorphic to a neighbourhood of
 $[(A,u)]$ in $\mathcal{D}^{*}_{k} (E) 
/ \hat{\mathcal{G}}_{k+1} (E)$. 
Moreover, from the elliptic regularity, the harmonic elements are actually 
smooth, 
therefore the neighbourhood of 
$[(A,u)]$ in $\mathcal{D}^{*}_{k} (E) 
/ \hat{\mathcal{G}}_{k+1} (E)$ 
is isomorphic to 
a neighbourhood of $[(A,u)]$ in $\mathcal{M}^{*} (E)$.

The assertions that $\kappa_{(A,u)}=0$ and the derivative of
 $\kappa_{(A,u)}$ at $0$ is zero just follow from the definition 
$\kappa_{(A,u)} = \psi \circ K_{(A,u)}^{-1}$ and the fact that the
 differential of $K_{(A,u)}$ at $0$ is the identity. 
\end{proof}

From Theorem \ref{kur}, one can deduce that 
$\mathcal{M}^{*}(E)$ is smooth around $[(A,u)]$ if 
$\bold{H}^{2}_{(A,u)} (Z ,\mathfrak{u}(E)) = 0$. 
But, as in the case of 
the Hermitian--Einstein connections (cf. \cite{MR888136}, 
\cite[Chap.~V$\!$I$\!$I \S 4]{MR909698}, 
\cite[Chap.~2 \S2.1]{itonak}, \cite[Chap.~4 \S4.5]{LT}), 
it can be improved in the following way.   
Firstly, we note that, 
corresponding to the decomposition of $\mathfrak{u}(r)
 = i \R \oplus \mathfrak{su} (r )$, 
the bundle $\mathfrak{u} (E)$ naturally decomposes into 
$\underline{\R}$ and $\mathfrak{u} (E)_{0}$ over $Z$,  
where $\mathfrak{u} (E)_{0}$ is the bundle of trace-free skew-Hermitian
endmorphisms of $E$, 
and there is a subcomplex of the complex \eqref{DTcomplex}, 
which is defined by
using the bundle $\mathfrak{u}(E)_{0}$ instead of $\mathfrak{u} (E)$. 
The decomposition is preserved by the operators of the complex, 
hence it induces a corresponding splitting of  
$H^{i}_{(A,u)} (Z, \mathfrak{u}(E)) \, (i = 0, 1, 2)$. 
For $(\alpha_c ,\upsilon_c) \in \Lambda^{1}(Z) \oplus A^{0,3} (Z)$, it
is always 
$H  
( \pi^{+} ( \alpha_c \wedge \alpha_c ) + \Lambda^2 [\upsilon_c 
, \bar{\upsilon}_c] 
+ \bar{*} \alpha_c \bar{*} \upsilon_c )=0 $, hence the
map $\kappa_{(A,u)}$ values in $H^{2} (Z , \mathfrak{u}(E)_{0})$.  
In particular, we obtain the following.

\begin{Corollary}
Around $[(A,u)] \in \mathcal{M}^{*} (E)$ with
 $H^{2}_{(A,u)} ( Z , \mathfrak{u}(E)_{0} )=0$ 
, the moduli space $\mathcal{M}^{*} (E)$ is smooth. 
\end{Corollary}

\begin{remark}
Around $(A,u) \in \mathcal{D}(E)$,  which is not irreducible, 
one can prove that 
$\bold{H}^{1}_{(A,u)} (Z, \mathfrak{u} (E))$ and 
$\bold{H}^{2}_{(A,u)} (Z, \mathfrak{u} (E))$ are 
 $\Gamma_{(A,u)}$-invariant, and the map  
$\kappa_{(A,u)}$ is $\Gamma_{(A,u)}$-equivariant. 
Hence, combining the claim in Remark \ref{resl}, one can deduce that 
around $[(A,u)]$ the moduli space $\mathcal{M} (E)$ is locally modeled on 
$\kappa_{(A,u)}^{-1} (0) / \Gamma_{(A,u)}$. 
\end{remark}

\section{The Hitchin--Kobayashi correspondence for the Donaldson--Thomas
 instantons on compact K\"{a}hler threefolds}
\label{sec:HKDT}

Perhaps one might ask what kind of a  
Hitchin--Kobayashi style correspondence would hold for the
Donaldson--Thomas instanton on compact K\"{a}hler threefolds. 
In this section, we describe this, 
which actually follows from a result by 
\'{A}lvarez-C\'{o}nsul and Garc\'{\i}a-Prada \cite{AG}.

Let $Z$ be a compact K\"{a}hler threefold, and $E = (E, h)$ a Hermitian
vector bundle over $Z$ with Hermitian metric $h$.  
If $(A,u)$ is a D--T instanton on $E$, 
then 
the connection $A$ defines a holomorphic structure $\bar{\partial}_{A}$
on $E$ as $F_{A}^{0,2} =0$, thus,  
we can think of $E$ as a locally free sheaf $\mathcal{O} (E , \dbar_{A})$. 
In addition,  
the $\text{End} (E)$-valued $(0,3)$-form $u$ is naturally identified
with a section of the bundle  
$\text{End}(E) \otimes K_{Z}^{-1}$, 
so $\bar{*} u$ is a section of the bundle  $\text{End}(E)
\otimes K_{Z}$.  
The equation $\dbar_{A}^{*} u =0$ implies 
$\dbar_{A} \bar{*} u =0$, hence, $\varphi := \bar{*} u$ is a holomorphic section of 
$\text{End} (E) \otimes K_Z$.

We then consider a pair $( \mathcal{E} , \varphi)$
consisting of a torsion-free sheaf $\mathcal{E}$ 
and a holomorphic section $\varphi$ of $\text{End} (\mathcal{E}) \otimes K_{Z}$. 
A subsheaf $\mathcal{F}$ of $\mathcal{E}$ 
is said to be a {\it $\varphi$-invariant} if 
$ \varphi  (\mathcal{F}) \subset \mathcal{F} \otimes K_{Z}$.   
We define a {\it slope} $\mu (\mathcal{F})$ of a coherent subsheaf $\mathcal{F}$ of
$\mathcal{E}$ by 
$ \mu (\mathcal{F}) 
:= \frac{1}{\text{rank} ( \mathcal{F} )} \int_{Z} 
c_1 ( \det \mathcal{F}) \wedge \omega^2 $.

\begin{definition}
A pair $(\mathcal{E},\varphi)$ 
consisting of a torsion-free sheaf $\mathcal{E}$ 
and a holomorphic section $\varphi$ of $\text{End} (\mathcal{E}) \otimes K_{Z}$ 
is called {\it semi-stable} if 
$ \mu ( \mathcal{F} ) \leq \mu ( \mathcal{E} ) $ 
for any 
$\varphi$-invariant coherent subsheaf 
$\mathcal{F}$ with $\text{rank} (\mathcal{F}) 
< \text{rank} (\mathcal{E})$.  
A pair $(\mathcal{E}, \varphi)$ is called {\it stable} if 
$ \mu ( \mathcal{F} ) < \mu ( \mathcal{E} ) $ 
for any 
$\varphi$-invariant coherent subsheaf 
$\mathcal{F}$ with $\text{rank} (\mathcal{F}) 
< \text{rank} (\mathcal{E})$.  
\label{def:stable}
\end{definition}

\begin{definition}
A pair $(\mathcal{E} , \varphi)$ 
consisting of a torsion-free sheaf $\mathcal{E}$ 
and a holomorphic section $\varphi$ of $\text{End} (\mathcal{E}) \otimes K_{Z}$ 
is said to be {\it poly-stable}
if it is a direct sum of stable sheaves with the same slopes in the sense of Definition
\ref{def:stable}. 
\end{definition}

Then the correspondence can be stated as a one-to-one
correspondence between a pair $(\mathcal{E} , \varphi)$, 
where $\mathcal{E}$ is a locally-free sheaf on a K\"{a}hler threefold
$Z$ and a holomorphic section $\varphi$ of $End(\mathcal{E}) \otimes K_{Z}$, which
is stable in the sense of  Definition \ref{def:stable}; and 
the existence of 
a solution to the Donaldson--Thomas equations on
$\mathcal{E}$. 
This fits into a setting studied by 
\'{A}lvarez-C\'{o}nsul and Garc\'{\i}a-Prada \cite{AG} (see also \cite{BGM}), 
and it is stated as a special case of their results as the case of a
twisted quiver bundle with one vertex and one arrow, whose head and tail
conincide, and with twisting sheaf the anti-canonical bundle. 
We state it in our setting as follows.

\begin{theorem}[\cite{AG}]
Let $Z$ be a compact K\"{a}hler threefold with K\"{a}hler form $\omega$. 
Let $(\mathcal{E} , \varphi)$ be a pair consisting of a locally-free
 sheaf $\mathcal{E}$ on $Z$ and a holomorphic section $\varphi$ of $\text{End}\, (\mathcal{E})
 \otimes K_{Z}$. 
Then, $(\mathcal{E} , \varphi)$ is poly-stable if and only if 
$\mathcal{E}$ admits a unique Hermitian metric $h$ satisfying 
$\Lambda F_{h} + \Lambda^3 [ \varphi , \bar{\varphi}^{h}] + 6 \pi i  
 \mu (\mathcal{E})  Id_{\mathcal{E}}  = 0 $, where $F_{h}$ is the curvature form of $h$, and
 $\Lambda := ( \wedge \omega )^{*}$. 
\label{th:HKDT}
\end{theorem}

Note that the equation $\bar{\partial}_{A}^{*} u =0$ in the Donaldson--Thomas equations on a compact K\"{a}hler threefold is implicitly addressed in Theorem \ref{th:HKDT} by saying that $\varphi = \bar{*} u$ is a holomorphic section of $\text{End} (\mathcal{E} ) \otimes K_{Z}$. 
One more remark is that a proof of the Hitchin--Kobayashi correspondence using the
Mehta--Ramanathan argument for the Vafa--Witten equations in \cite{Tan}
could also apply to the Donaldson--Thomas instanton on smooth projective
threefold as mentined in \cite{Tan}.

%%%%%%%%%%%%%%%%%%%%%%%%%%%%%%%%%%%%%%%%%%%%%%%%%%%%%%%%%%%%%%%%%%%%%%%%%%%%%%%%%%%%%%%%%%%%%%%%%%%%%%

\begin{flushleft}
Graduate School of Mathematics, Nagoya University, 
Furo-cho, Chikusa-ku, Nagoya, 464-8602, Japan \\
yu2tanaka@gmail.com
\end{flushleft}

\end{document}